\documentclass{article}
\usepackage{xcolor,hyperref}
\usepackage{float}
\usepackage{amsmath,amssymb,bm,stmaryrd}
\usepackage{amsthm}
\usepackage{ascmac}
\usepackage{array}
\usepackage[margin=25truemm]{geometry}
\allowdisplaybreaks[4]
\usepackage[pdftex]{graphicx}

\theoremstyle{plain}
\newtheorem{thm}{Theorem}

\theoremstyle{definition}

\newtheorem{rmk}[thm]{Remark}

\newcommand{\Zz}{\mathbb{Z}}
\newcommand{\Cz}{\mathbb{C}}

\newcommand{\txi}[1]{\tilde{\xi}_{#1}}

\begin{document}
\title{A solution in terms of mock modular forms for the $q$-Painlev\'{e} equation of the type $(A_2+A_1)^{(1)}$}
\author{S. Tsuchimi}
\date{}

\maketitle

\begin{abstract}
We present a solution of the $(A_2+A_1)^{(1)}$ $q$-Painlev\'{e} equation in terms of the $\mu$-function. 
The $\mu$-function introduced by Zwegers is the most fundamental object in the study of mock theta functions. 
The results of this paper give us an expectation that the theory of mock modular forms and the $\tau$-functions of discrete integrable systems are closely related. 
\end{abstract}

\section{Introduction}
Throughout this paper, let  $\tau\in\Cz$ be a complex number ${\rm Im}(\tau)>0$, and $q:=e^{\pi i\tau}$. We define the $q$-factorials and Jacobi theta function as follows:
\begin{align}
(q)_n&:=\prod_{l=1}^n(1-q^l)\quad (n\in\Zz_{>0}),\quad (q)_0:=1,\\
\vartheta(z;\tau)&=\vartheta(z):=\sum_{\nu\in\Zz+\frac{1}{2}}e^{2\pi i\nu \left(z+\frac{1}{2}\right)+\pi i\nu^2\tau}.
\end{align}

Mock theta functions were introduced by Ramanujan in $1920$. 
Ramanujan discovered interesting functions that have asymptotic behavior similar to ``theta functions'' (i.e., modular forms) at roots of unity but are not ``theta functions'', and called these functions mock theta functions.  
Starting with Watson \cite{W}, modular transformation laws of mock theta functions have been studied for a long time, and were finally resolved by Zwegers \cite{Z}. 
One of Zwegers' greatest results is the introduction of the $\mu$-function, which later gave a prototype for harmonic Maa{\ss}-Jacobi forms, to uniformly capture the modular transformation laws of mock theta functions
\begin{align*}
\mu(u,v\mid \tau)=\mu(u,v):=\frac{e^{\pi iu}}{\vartheta(v)}\sum_{n\in\Zz}\frac{(-1)^ne^{2\pi inv+\pi in(n+1)\tau}}{1-e^{2\pi iu+2\pi in\tau}},\quad u,v\not\in\Lambda_\tau:=\{m\tau+n\in\Cz;m,n\in\Zz\}.
\end{align*}
Zwegers showed that the $\mu$-function has many interesting symmetric properties. 
As a property of particular importance, the $\mu$-function gives a function $\tilde{\mu}$ satisfying a transformation law like Jacobi forms by adding an appropriate non-holomorphic function to the $\mu$-function: 
\begin{align}
   \tilde{\mu}(u,v\mid\tau)
   &:=
   \mu(u,v\mid\tau)+\frac{i}{2}R(u-v\mid\tau), \nonumber\\
\label{mu modular}
\tilde{\mu}(u,v\mid\tau)&=e^{\frac{\pi i}{4}}\tilde{\mu}(u,v\mid\tau+1),\quad \tilde{\mu}(u,v\mid\tau)=\frac{i}{\sqrt{-i\tau}}e^{-\pi i\frac{(u-v)^2}{\tau}}\tilde{\mu}\left(\frac{u}{\tau},\frac{v}{\tau}\mid-\frac{1}{\tau}\right),
\end{align}
where
\begin{align*}
R(u\mid \tau)
   &:=
   \sum_{\nu\in\Zz+\frac{1}{2}}\left\{{\rm sgn}(\nu)-E((\nu+a)\sqrt{2t})\right\}(-1)^{\nu-\frac{1}{2}}e^{-2\pi i\nu u}q^{-\nu^2}, \quad
   E(x)
   :=
   2\int_0^x e^{-\pi z^2}dz
\end{align*}
and $t={\rm Im}(\tau), a=\frac{{\rm Im}(u)}{{\rm Im}(\tau)}$. 
This result made an epoch in the study of mock modular forms. 
Since the $\mu$-function is generalization of Ramanujan's mock theta functions, nowadays the mock theta functions are defined as a holomorphic part of harmonic Maa{\ss}  forms \cite[Definition 5.16]{BFOR}.


Recently, Shibukawa and the author introduced a natural extension of the $\mu$-function in terms of $q$-difference equations \cite[Definition 1.1]{ST}. 
\begin{align}
&\mu(u,v;\alpha\mid\tau)=\mu(u,v;\alpha):=\frac{e^{\pi i\alpha(u-v)}}{\vartheta(v)}\sum_{n\in\Zz}(-1)^ne^{2\pi i(n+\frac{1}{2})v}e^{\pi in(n+1)\tau}\prod_{j=1}^\infty\frac{1-e^{2\pi iu+2\pi i(n+j)\tau}}{1-e^{2\pi iu+2\pi i(n-\alpha+j)\tau}}. 
\end{align}
Note that the function $\mu(u,v;\alpha)$ is certainly a one parameter extension of the $\mu$-function, since $\mu(u,v;1)=\mu(u,v)$. 
The function $\mu(u,v;\alpha)$ satisfies the $q$-Hermite-Weber equation with respect to the variables $u$ and $v$, and satisfies the $q$-Bessel equation for the valiable $\alpha$ \cite[equations (1.23) and (1.31)]{ST}. 
Also, the function $\mu(u,v;\alpha)$ has many of the same symmetries as the $\mu$-function. 
A particularly interesting point is that the function $\mu(u,v;\alpha)$ for $\alpha=n\in\Zz_{\geq 0}$ is expressed using the $\mu$-function as follows:
\begin{align*}
\mu(u,v;n+1)
   =
   F_{n+1}(u-v\mid q^2)\mu(u,v)-iq^{-\frac{1}{4}}\sum_{l=1}^n\frac{q^{2l}}{(q^2)_l}F_{n-l+1}(u-v\mid q^2)H_{l-1}(u-v\mid q^2), 
\end{align*}
where 
\begin{align*}
H_n(u\mid q):=\sum_{l=0}^n\frac{(q)_n}{(q)_l(q)_{n-l}}e^{\pi i(n-2l)u},\quad F_{n+1}(u\mid q):=(-1)^nq^\frac{n(n+1)}{2}\sum_{l=0}^n\frac{q^{l(l-n)}}{(q)_{n-l}(q)_l}e^{\pi i(n-2l)u}. 
\end{align*}
That is, the function $\mu(u,v;n)$ is essentially equivalent to the original $\mu$-function, and gives its modular completion $\tilde{\mu}(u,v\mid n)$ by adding an appropriate non-holomorphic function
\begin{align*}
\frac{i}{2}R_n(u-v\mid q):=\frac{i}{2}F_n(u-v\mid q^2)R(u-v;\tau)+iq^{-\frac{1}{4}}\sum_{l=1}^{n-1}\frac{q^{2l}}{(q^2)_{l}}F_{n-l}(u-v\mid q^2)H_{l-1}(u-v\mid q^2)
\end{align*}
to the function $\mu(u,v;n)$. 
In this paper, we present the function $\mu(u,v;\alpha)$ is closely related to a discrete Painlev\'{e} equation. 

Discrete Painlev\'{e} equations are introduced from various fields such as physics and discrete integrable systems (for examples \cite{BK}, \cite{FIK}, \cite{PS}, \cite{RGH}). 
They were classified by H. Sakai in $2001$ in a geometric framework based on the type of rational surfaces \cite{S}. 
Among the classified equations, multiplicative discrete Painlev\'{e} equations are sometimes called the $q$-Painlev\'{e} equations. 
For example, let $W$ be the extended affine Weyl group $W:={\rm Aut}(A_5^{(1)})\ltimes W((A_2+A_1)^{(1)})$. That is, the group $W$ has generators $s_0,s_1,s_2,r_0,r_1,\iota,\pi$ satisfying the following relations: 
\begin{align*}
s_j^2=(s_js_{j+1})^3=r_k^2=\iota^2=\pi^6=1,\quad \pi s_j=s_{j+1}\pi,\quad r_k\pi=\pi r_{k+1},\quad r_0\iota=\iota r_{1},\quad (j\in\Zz/3\Zz,k\in\Zz/2\Zz). 
\end{align*}
The following representations of the group $W$ is constructed a representation of the group $W$ on the field $K:=\Cz\left(q,x,a_0,a_1,a_2,f_0,f_1,f_2\right)$ by the following actions (for example, see \cite{T}): 
\begin{align*}
s_j(a_j)
    &=
    \frac{1}{a_j},\quad s_j(a_{j+1})=a_ja_{j+1},\quad s_j(a_{j+2})=a_ja_{j+2},\\
 s_j(f_{j+1})&=f_{j+1}\frac{a_j+f_j}{1+a_jf_j},\quad s_j(f_{j+2})=f_{j+2}\frac{1+a_jf_j}{a_j+f_j},\\
r_0(x)&=\frac{q^2}{x},\quad
r_0(f_j)
    =
    \frac{a_ja_{j+1}}{f_{j+2}}\frac{a_ja_{j+2}+a_{j+2}f_j+f_jf_{j+2}}{a_ja_{j+1}+a_jf_{j+1}+f_jf_{j+1}},\quad\\
r_1(x)&=\frac{1}{x},\quad
r_1(f_j)=\frac{1}{a_ja_{j+1}f_{j+1}}\frac{1+a_jf_j+a_ja_{j+1}f_jf_{j+1}}{1+a_{j+2}f_{j+2}+a_ja_{j+2}f_jf_{j+2}},\\
\iota(q)&=\frac{1}{q},\quad \iota(x)=xq,\quad \iota(a_j)=\frac{1}{a_{2j}},\quad \iota(f_j)=f_{2j},\\
\pi(x)&=\frac{q}{x},\quad \pi(a_j)=a_{j+1},\quad \pi(f_j)=\frac{1}{f_{j+1}},\quad (j\in\Zz/3\Zz).
\end{align*}
The group $W$ has a subgroup isomorphic to $\Zz^3$ whose generators are $T:=\pi^3r_0, T_1:=\pi^2s_2s_0,T_2:=\pi^4s_1s_0$.
If $a_0a_1a_2=q,f_0f_1f_2=x^2q$, these generators satisfy the following actions:
\begin{align}
\label{shift}
&T(x,a_1,a_2)=(x q,a_1,a_2),\quad T_1(x,a_1,a_2)=\left(x,a_1q,a_2\right),\quad T_2(x,a_1,a_2)=\left(x,a_1,a_2q\right)\\
\label{tx}
&T(f_j)=a_ja_{j+1}f_{j+1}\frac{1+a_{j+2}f_{j+2}+a_{j+2}a_{j}f_{j+2}f_j}{1+a_jf_j+a_ja_{j+1}f_jf_{j+1}},\quad (j\in\Zz/3\Zz)\\
\label{t1}
&\frac{T_1(f_2)}{f_1}=\frac{1+a_1f_1+a_1^2a_2f_2+a_1a_2f_1f_2}{a_1a_2+a_1f_1+a_1^2a_2f_1+f_1f_2},\quad \frac{T_1^{-1}(f_1)q}{f_2}=\frac{x^2q^2+a_1a_2f_1f_2}{x^2a_1a_2+f_1f_2}\\
\label{t2}
&\frac{T_2(f_1)}{f_2}=\frac{a_!a_2^2+f_1+a_1a_"f_2+a_2f_1f_2}{a_2+a_1a_2f_1+f_2+a_1a_2^2f_1f_2},\quad \frac{T_2^{-1}(f_2)}{f_1q}=\frac{a_1a_2x^2+f_1f_2}{x^2q^2+a_1a_2f_1f_2}
\end{align}
The equations (\ref{shift})-(\ref{t2}) are called the type $(A_2+A_1)^{(1)}$ $q$-Painlev\'{e} equations. 
It is known that the type $(A_2+A_1)^{(1)}$ $q$-Painlev\'{e} equation has $q$-hypergeometric solutions satisfying the $q$-Bessel and $q$-Hermite-Weber equation
\begin{align*}
[axT^2-(1-x)T-1]f(x)=0\quad(\text{$q$-HW eq.}) ,\quad \left[T-\left(q^\frac{\nu}{2}+q^{-\frac{\nu}{2}}\right)T^\frac{1}{2}+\left(1+\frac{x^2}{4}\right)\right]f(x)=0\quad (\text{$q$-B eq.}),
\end{align*}
when $a_0=1, f_0=-1$ (for details, see \cite{KK}, \cite{KMNOY1}, \cite{KMNOY2}, \cite{KNY}, \cite{RGTT}, etc). 

The main result of this paper is the following theorem. 
\begin{thm}
\label{thm:1}
Putting $k,m\in\Cz,n\in\Zz_{\geq0}$ and 
\begin{align}
\label{xi}
\xi_{m,n,k}:=\det\left[\mu(u+v+k\tau,v;m-j-j')\right]_{j,j'=0}^n. 
\end{align}
Then, 
\begin{align}
\label{solution}
(x,a_1,a_2;f_1,f_2)=\left(x,-q^m,q^{-n};\frac{\xi_{m,n,1}}{\xi_{m,n,0}}\frac{\xi_{m,n-1,0}}{\xi_{m,n-1,1}},-\frac{\xi_{m,n-1,1}}{\xi_{m,n-1,0}}\frac{\xi_{m-1,n-1,0}}{\xi_{m-1,n-1,1}}\right)
\end{align}
are solutions of the type $(A_2+A_1)^{(1)}$ $q$-Painlev\'e equation, where $x:=e^{\pi iu}$. 
\end{thm}

Note that the parameter $v$ is an independent parameter of the equation, so the solution $(\ref{solution})$ is a one parameter family of solutions of the type $(A_2+A_1)^{(1)}$ $q$-Painlev\'{e} equation. 
Furthermore, we obtain the following remark by applying the action $s_2$ to the solution $(\ref{solution})$. 
\begin{rmk}
Putting $k,m\in\Cz,n\in\Zz_{\geq0}$, then 
\begin{align*}
(x,a_0,a_2;f_0,f_2)=\left(x,-q^{-m},q^{n+1};\frac{\xi_{m,n,1}\xi_{m,n-1,0}}{\xi_{m,n,0}\xi_{m,n-1,1}},-\frac{\xi_{m+1,n,1}\xi_{m,n,0}}{\xi_{m+1,n,0}\xi_{m,n,1}}\right)
\end{align*}
are solutions of the type $(A_2+A_1)^{(1)}$ $q$-Painlev\'e equation. 
\end{rmk}
The $q$-hypergeometric solutions of the $q$-Painlev\'{e} equations are obtained from rewriting the equations as a kind of Riccati type $q$-difference equations
\begin{align*}
f(xq)=\frac{af(x)+b}{cf(x)+d}
\end{align*}
by specializing the parameters, and linearizing them. 
Since the function $\mu(u,v;\alpha)$ satisfies the $q$-Hermite-Weber equation and $q$-Bessel equation, it is expected that the function $\mu(u,v;\alpha)$ is related to a solution of the type $(A_2+A_1)^{(1)}$ $q$-Painlev\'{e} equation.
The fact stated in Theorem $\ref{thm:1}$ is a novel claim and an important result that connects discrete integrable systems and the theory of mock modular forms. 

In the next section, we present and prove bilinear relations satisfied the function $\xi_{m,n,k}$. 
Then, we prove Theorem \ref{thm:1}, since the relations satisfied by the function $\xi_{m,n,k}$ coincide with bilinear relations satisfied by the $\tau$-function of the type $(A_2+A_1)^{(1)}$ $q$-Painlev\'{e} equation. 

\section{Proof of the main result}
\label{section 3}
In this section, we prove Theorem $\ref{thm:1}$. 
First, we present bilinear relations satisfied by the function $\xi_{m,n,k}$.

\begin{thm}
\label{bi linear}
The functions $\xi_{m,n,k}$ defined by equation $(\ref{xi})$ satisfy the following bilinear equations: 
\begin{align}
\label{1}
& x^2q^{m+2k}\xi_{m-1,n-1,k+1}\xi_{m,n,k}-x q^{m-n+k}\xi_{m,n-1,k+1}\xi_{m-1,n,k}+q^n\xi_{m,n,k+1}\xi_{m-1,n-1,k}=0\\
\label{2}
& x^2 q^{2n+2k-m}\xi_{m-1,n-1,k}\xi_{m,n,k-1}-x q^k\xi_{m-1,n,k}\xi_{m,n-1,k-1}+q^n\xi_{m,n,k}\xi_{m-1,n-1,k-1}=0\\
\label{3}
& x q^{m+k-n}\xi_{m-1,n-2,k}\xi_{m,n,k-1}-q^{n}\xi_{m-1,n-1,k}\xi_{m,n-1,k-1}+\xi_{m,n-1,k}\xi_{m-1,n-1,k-1}=0\\
\label{4}
& x\xi_{m-1,n-1,k+1}\xi_{m,n-1,k}-xq^{n}\xi_{m,n-1,k+1}\xi_{m-1,n-1,k}+q^{m-n-k}\xi_{m-1,n-2,k}\xi_{m,n,k+1}=0\\
\label{5}
& x^2 q^{2k+m}\xi_{m-2,n-1,k+1}\xi_{m,n,k}-x q^{k+m-n}\xi_{m-1,n-1,k+1}\xi_{m-1,n,k}+\xi_{m,n,k+1}\xi_{m-2,n-1,k}=0\\
\label{6}
& x^2 q^{2k-m}\xi_{m-2,n-1,k}\xi_{m,n,k-1}-x q^{k-n}\xi_{m-1,n,k}\xi_{m-1,n-1,k-1}+\xi_{m,n,k}\xi_{m-2,n-1,k-1}=0\\
\label{7}
& x q^{k+2m-n}\xi_{m,n,k}\xi_{m-1,n-2,k}-q^n\xi_{m,n-1,k}\xi_{m-1,n-1,k}+\xi_{m,n-1,k+1}\xi_{m-1,n-1,k-1}=0\\
\label{8}
& xq^k\xi_{m,n-1,k-1}\xi_{m-1,n-1,k+1}-xq^{n+k}\xi_{m-1,n-1,k}\xi_{m,n-1,k}+q^{2m-n}\xi_{m,n,k}\xi_{m-1,n-2,k}=0.
\end{align}
\end{thm}
\begin{proof}
We first show the case $k=0$ (or $k=\pm1$) and then prove to put $x\mapsto xq^k$. Putting $\nu_j:=\mu(u+v,v;m-j)$. 
Since the function $\mu(u,v;\alpha)$ satisfies the following relations \cite{ST}: 
\begin{align}
\label{xuad}
&q^{-\alpha}\mu(u+v+\tau,v;\alpha)+x^2\mu(u+v,v;\alpha)-x\mu(u+v,v;\alpha-1)=0\\
\label{xdad}
&\mu(u+v,v;\alpha)+x^2q^{-\alpha}\mu(u+v-\tau,v;\alpha)-x\mu(u+v,v;\alpha-1)=0, 
\end{align}
we obtain 
\begin{align}
\xi_{m-1,n-1,\pm1}=x^{\pm n}q^{n(m-n)}\det
\begin{bmatrix}
       1       & x^{\pm 1}  & \ldots &       x^{\pm n}     \\
 \nu_{1} & \nu_{2} & \ldots & \nu_{n+1} \\
 \nu_{2} & \nu_{3} & \ldots & \nu_{n+2}\\
\vdots & \vdots & \ddots & \vdots \\
 \nu_{n} & \nu_{n+1} & \ldots & \nu_{2n}
\end{bmatrix}.
\end{align}
Hence we have
\begin{align*}
\text{{\rm LHS} of $(\ref{1})_{k=0}$}=&x^{n+1}q^{mn-n^2+m}\det
\begin{bmatrix}
 x          & x^2 & \ldots &       x^{n+1}     \\
\nu_{1}      & \nu_{2} & \ldots & \nu_{n+1} \\
\vdots & \vdots & \ddots & \vdots \\
\nu_{n} & \nu_{n+1} & \ldots & \nu_{2n}
\end{bmatrix}\det[\nu_{j+j'}]_{j,j'=0}^n\\
&-x^{n+1}q^{mn-n^2+m}\det
\begin{bmatrix}
 1     & x  & \ldots &       x^n     \\
\nu_0 & \nu_{1} & \ldots & \nu_{n} \\
\vdots & \vdots & \ddots & \vdots \\
\nu_{n-1} & \nu_{n} & \ldots & \nu_{2n-1}
\end{bmatrix}\det[\nu_{j+j'+1}]_{j,j'=0}^n\\
&+x^{n+1}q^{mn-n^2+m}\det
\begin{bmatrix}
 1          & x  & \ldots &       x^{n+1}     \\
\nu_0   & \nu_{1} & \ldots & \nu_{n+1} \\
\vdots & \vdots & \ddots & \vdots \\
\nu_{n} & \nu_{n+1} & \ldots & \nu_{2n+1}
\end{bmatrix}\det[\nu_{j+j'+1}]_{j,j'=0}^{n-1}.
\end{align*}
Comparing the coefficients of $x^l$, we obtain $\text{{\rm LHS} of $(\ref{1})_{k=0}$}=0$ using the Pl\"{u}cker relation. From similar calculations, we also obtain equations $(\ref{2})$--$(\ref{6})$. 

Next, note that 
\begin{align}
\label{00}
(x^{-1}-x)\xi_{m-1,n-1,0}=q^{2n(m-n-1)}\det
\begin{bmatrix}
1 & x & \ldots & x^{n+1}\\
1 & x^{-1} & \ldots & x^{-n-1}\\
\nu_1 & \nu_2 & \ldots & \nu_{n+2}\\
\vdots & \vdots & \ddots & \vdots \\
\nu_n & \nu_{n+1} & \ldots & \nu_{2n+1}
\end{bmatrix}. 
\end{align}
In fact, since 
\begin{align*}
\xi_{m-1,n-1,0}&=\left(\frac{x}{q}\right)^nq^{n(m-n)}\det
\begin{bmatrix}
1 & \ldots & (x/q)^n\\
\mu(u+v-\tau,v;m-1) & \ldots & \mu(u+v-\tau,v;m-n-1)\\
\vdots & \ddots & \vdots \\
\mu(u+v-\tau,v;m-n) & \ldots & \mu(u+v-\tau,v;m-2n)
\end{bmatrix}\\
&=q^{2n(m-n-1)}\det
\begin{bmatrix}
1 & \ldots & x^n\\
\nu_2-\frac{\nu_1}{x} & \ldots & \nu_{n+2}-\frac{\nu_n+1}{x}\\
 & \ldots & \\
\nu_{n+1}-\frac{\nu_n}{x} & \ldots & \nu_{2n+1}-\frac{\nu_{2n}}{x}
\end{bmatrix}, 
\end{align*}
we put $X_j:=\sum_{k=0}^jx^{2k-j}$ which implies $X_j-x^{-1}X_{j-1}=x^j$, so we have
\begin{align*}
\xi_{m-1,n-1,0}&=q^{2n(m-n-1)}\det
\begin{bmatrix}
X_0-0 & X_1-\frac{X_0}{x} & \ldots & X_n-\frac{X_{n-1}}{x}\\
\nu_2-\frac{\nu_1}{x} & \nu_3-\frac{\nu_2}{x} & \ldots & \nu_{n+2}-\frac{\nu_{n+1}}{x}\\
\vdots & \vdots & \ddots & \vdots \\
\nu_{n+1}-\frac{\nu_n}{x} & \nu_{n+2}-\frac{\nu_{n+1}}{x} & \ldots & \nu_{2n+1}-\frac{\nu_{2n}}{x}
\end{bmatrix}\\
 &=
q^{2n(m-n-1)}\sum_{j=0}^{n+1}\det
\begin{bmatrix}
0 & -X_0/x & \ldots & -X_{j-2}/x & X_{j} & \ldots & X_n\\
-\nu_1/x & -\nu_2/x & \ldots & -\nu_{j}/x & \nu_{j+2} & \ldots & \nu_{n+2}\\
\vdots & \vdots & \ddots & \vdots & \vdots & \ddots & \vdots \\
-\nu_n/x & -\nu_{n+1}/x & \ldots & -\nu_{n+j-1}/x & \nu_{n+j+1} & \ldots & \nu_{2n+1}
\end{bmatrix}\\
&=q^{2n(m-n-1)}\det
\begin{bmatrix}
1 & x^{-1} & \ldots & x^{-n-1}\\
0 & X_0 & \ldots & X_n\\
\nu_1 & \nu_2 & \ldots & \nu_{n+2}\\
\vdots & \vdots & \ddots & \vdots \\
\nu_n & \nu_{n+1} & \ldots & \nu_{2n+1}
\end{bmatrix}\\
&=\frac{q^{2n(m-n-1)}}{x-x^{-1}}\det
\begin{bmatrix}
1 & x^{-1} &\ldots & x^{-n-1}\\
0 & x-x^{-1} & \ldots & x^{n+1}-x^{-n-1}\\
\nu_1 & \nu_2 & \ldots & \nu_{n+2}\\
\vdots & \vdots & \ddots & \vdots \\
\nu_n & \nu_{n+1} & \ldots & \nu_{2n+1}
\end{bmatrix}. 
\end{align*}
Therefore, the equation $(\ref{00})$ holds. Using the equation $(\ref{00})$, we have 
\begin{align*}
    &\xi_{m,n-1,1}\xi_{m-1,n-1,-1}-\xi_{m-1,n-1,1}\xi_{m,n-1,-1}+(x-x^{-1})q^{2m-n}\xi_{m,n,0}\xi_{m-1,n-2,0}\\
    &=
    q^{n(2m-2n+1)}\det
\begin{bmatrix}
1 & \ldots & x^n\\
\nu_0 & \ldots & \nu_n\\
 & \ldots & \\
 \nu_{n-1} & \ldots & \nu_{2n-1}
\end{bmatrix}\det
\begin{bmatrix}
1 & \ldots & x^{-n}\\
\nu_1 & \ldots & \nu_{n+1}\\
 & \ldots & \\
\nu_n & \ldots & \nu_{2n}
\end{bmatrix}\\
    &\quad-q^{n(2m-2n+1)}\det
\begin{bmatrix}
1 & \ldots & x^n\\
\nu_1 \ldots & \nu_{n+1}\\
 & \ldots & \\
\nu_n & \ldots & \nu_{2n}
\end{bmatrix}\det
\begin{bmatrix}
1 & \ldots & x^{-n}\\
\nu_0 & \ldots & \nu_n\\
 & \ldots & \\
\nu_{n-1} & \ldots & \nu_{2n-1}
\end{bmatrix}+(x-x^{-1})q^{2m-n}\xi_{m,n,0}\xi_{m-1,n-2,0}\\
     &=
     q^{n(2m-2n+1)}\det
\begin{bmatrix}
\nu_0 & \ldots & \nu_n\\
 & \ldots & \\
\nu_n & \ldots & \nu_{2n} 
\end{bmatrix}\det
\begin{bmatrix}
1 & \ldots & x^n\\
1 & \ldots & x^{-n}\\
\nu_1 & \ldots & \nu_{n+1}\\
 & \ldots & \\
\nu_{n-1} & \ldots & \nu_{2n-1}
\end{bmatrix}-(x^{-1}-x)q^{2m-n}\xi_{m,n,0}\xi_{m-1,n-2,0}\\
    &=0, 
\end{align*}
and similar calculations, we have
\begin{align*}
&x^{-1}\xi_{m,n-1,1}\xi_{m-1,n-1,-1}-x\xi_{m-1,n-1,1}\xi_{m,n-1,-1}+(x-x^{-1})q^n\xi_{m,n-1,0}\xi_{m-1,n-1,0}=0. 
\end{align*}
Therefore, From 
\begin{align*}
\begin{bmatrix}
1 & -1\\
x & -x^{-1}
\end{bmatrix}
\begin{bmatrix}
\xi_{m-1,n-1,1}\xi_{m,n-1,-1}\\
\xi_{m,n-1,1}\xi_{m-1,n-1,-1}
\end{bmatrix}=(x-x^{-1})
\begin{bmatrix}
\xi_{m,n,0}\xi_{m-1,n-2,0}\\
q^n\xi_{m,n-1,0}\xi_{m-1,n-1,0}
\end{bmatrix}, 
\end{align*}
we have equations $(\ref{7})$ and $(\ref{8})$. 
\end{proof}

Next, we present the following fact that solutions of the type $(A_2+A_1)^{(1)}$ $q$-Painlev\'{e} equation are constructed by using functions satisfying certain bilinear forms. 
The following theorem have been implicitly known since a long time ago (for examples, \cite{KNY2}, \cite{T}, \cite{N}). 
There does not seem to be any paper that gives full details of this, so we give them properly in this paper. 

\begin{thm}
\label{tau thm}
Suppose that a function $\xi=\xi(a_1,a_2,x)$ satisfies the following bilinear relations:
\begin{align}
\label{11}
0
    &=
     \frac{q^{a-b-c}}{a_1a_2}\txi{a+1,b+1,c}\txi{a,b,c+1}-x^2q^c\txi{a+1,b+1,c+1}\txi{a,b,c}+a_2xq^b\txi{a+1,b,c}\txi{a,b+1,c+1}
     \\
\label{12}
    0&=
    \frac{x^2q^{a-b+2c}}{a_1a_2}\txi{a+1,b+1,c}\txi{a,b,c-1}-\txi{a+1,b+1,c-1}\txi{a,b,c}+a_2xq^{b+c}\txi{a+1,b,c}\txi{a,b+1,c-1}
    \\
\label{13}
    0&=
    a_1a_2xq^{b+c}\txi{a+1,b+2,c}\txi{a,b,c-1}-q^a\txi{a+1,b+1,c-1}\txi{a,b+1,c}+\frac{q^{a-b}}{a_2}\txi{a+1,b+1,c}\txi{a,b+1,c-1}
    \\
\label{14}
    0&=
    a_1a_2q^{b-c}\txi{a+1,b+2,c}\txi{a,b,c+1}-xq^{a}\txi{a+1,b+1,c+1}\txi{a,b+1,c}+\frac{xq^{a-b}}{a_2}\txi{a+1,b+1,c}\txi{a,b+1,c+1}
    \\
\label{15}
    0&=
    q^{a-c}\txi{a+2,b+1,c}\txi{a,b,c+1}-a_1x^2q^c\txi{a+2,b+1,c+1}\txi{a,b,c}+a_1a_2xq^b\txi{a+1,b+1,c+1}\txi{a+1,b,c}
    \\
\label{16}
    0&=
    x^2q^{a+c}\txi{a+2,b+1,c}\txi{a,b,c-1}-a_1q^{-c}\txi{a+2,b+1,c-1}\txi{a,b,c}+a_1a_2xq^{b}\txi{a+1,b+1,c-1}\txi{a+1,b,c}
    \\
\label{17}
    0&=
    a_1a_2x\txi{a,b,c}\txi{a+1,b+2,c}-\frac{q^{2a-2b-c}}{a_1a_2}\txi{a+1,b+1,c}\txi{a,b+1,c}+\frac{q^{2a-b-c}}{a_1}\txi{a,b+1,c+1}\txi{a+1,b+1,c-1}
    \\
\label{18}
    0&=
    a_1a_2\txi{a,b,c}\txi{a+1,b+2,c}-\frac{xq^{2a-2b+c}}{a_1a_2}\txi{a+1,b+1,c}\txi{a,b+1,c}+\frac{xq^{2a+2b+c}}{a_1}\txi{a,b+1,c-1}\txi{a+1,b+1,c+1}
\end{align}
where $\txi{a,b,c}:=T_1^{-a}T_2^bT^c\xi$. 
Then, the pair
\begin{align*}
(x,a_1,a_2;f_1,f_2)=\left(x,a_1,a_2;\frac{\txi{0,0,1}\txi{0,1,0}}{\txi{0,0,0}\txi{0,1,1}},-\frac{\txi{0,1,1}\txi{1,1,0}}{\txi{0,1,0}\txi{1,1,1}}\right)
\end{align*}
gives a solution of the type $(A_2+A_1)^{(1)}$ $q$-Painlev\'{e} equation $(\ref{shift})$-$(\ref{t2})$.
\end{thm}
\begin{rmk}
If $\xi$ is a solution of the bilinear relations $(\ref{11})$-$(\ref{18})$, $a_1^\alpha a_2^\beta x^\gamma\delta\xi$ is also a solution. 
\end{rmk}
\begin{proof}[Proof of Theorem $\ref{tau thm}$]
Putting
\begin{align*}
(x,a_1,a_2;f_1,f_2)=\left(x,a_1,a_2;\frac{\txi{0,0,1}\txi{0,1,0}}{\txi{0,0,0}\txi{0,1,1}},-\frac{\txi{0,1,1}\txi{1,1,0}}{\txi{0,1,0}\txi{1,1,1}}\right),
\end{align*}
we show
\begin{align}
\label{Tf1}
f_1T(f_1)&=\frac{a_1a_2f_1f_2+a_1x^2q^2f_1+x^2q^2}{a_1a_2f_1f_2+a_1f_1+1}\\
\label{Tf2}
\frac{f_2T(f_2)}{x^2q^2}&=\frac{a_1a_2f_1f_2+a_1f_1+1}{a_1a_2f_1f_2+a_1f_1+x^2q^2}\\
\label{t1f1}
\frac{T_1^{-1}(f_1)q}{f_2}&=\frac{a_1a_2f_1f_2+x^2q^2}{f_1f_2+a_1a_2x^2}\\
\label{t1f2}
\frac{f_1f_2T_1^{-1}(f_2)}{x^2}&=\frac{a_1^2a_2f_1f_2^2+q^2f_1f_2+a_1x^2q^2f_2+a_1a_2x^2q^2}{a_1a_2f_1f_2^2+a_1f_1f_2+x^2q^2f_2+a_1^2a_2x^2}\\
\label{t2f1}
\frac{f_1f_2T_2^{-1}(f_1)}{x^2}&=\frac{q^2f_1^2f_2+a_1a_2^2f_1f_2+a_1a_2x^2q^2f_1+a_2x^2q^2}{a_2f_1^2f_2+a_1a_2f_1f_2+a_1a_2^2x^2f_1+x^2q^2}\\
\label{t2f2}
\frac{T_2^{-1}(f_2)}{qf_1}&=\frac{f_1f_2+a_1a_2x^2}{a_1a_2f_1f_2+x^2q^2}. 
\end{align}
From equations $(\ref{11}), (\ref{14})$ and $(\ref{17})$, note that 
\begin{align}
0&=
    a_1a_2x\txi{0,0,0}\txi{1,2,0}-\frac{1}{a_1a_2}\txi{1,1,0}\txi{0,1,0}+\frac{1}{a_1}\txi{0,1,1}\txi{1,1,-1}
    \nonumber\\
    &=x^2\frac{\txi{0,0,0}}{\txi{0,0,1}}\left(\txi{0,1,0}\txi{1,1,1}-\frac{1}{a_2}\txi{0,1,1}\txi{1,1,0}\right)+\frac{1}{a_1}\txi{0,1,1}\txi{1,1,-1}-\frac{1}{a_1a_2}\txi{0,1,0}\txi{1,1,0}
    \nonumber\\
    &=\frac{a_1a_2x^2\txi{0,0,0}\txi{1,1,1}-\txi{0,0,1}\txi{1,1,0}}{a_1^2a_2^2x\txi{1,0,0}\txi{0,0,1}}\left(a_1a_2x\txi{0,1,0}\txi{1,0,0}-\frac{a_1x^2}{a_2}\txi{0,0,0}\txi{1,1,0}+\txi{0,0,1}\txi{1,1,-1}\right)
    \nonumber\\
    \label{110}
    &=\frac{\txi{0,1,1}}{a_1a_2\txi{0,0,1}}\left(a_1a_2^2x\txi{0,1,0}\txi{1,0,0}-\txi{0,0,0}\txi{1,1,0}+a_2\txi{0,0,1}\txi{1,1,-1}\right).
\end{align}
Using equations $(\ref{12}), (\ref{13}), (\ref{17})$ and $(\ref{110})$, we have

\begin{align*}
\text{{\rm RHS} of (\ref{Tf1})}
 &=
 \frac{a_1x^2q^{2}\txi{0,0,1}\txi{0,1,0}\txi{1,1,1}+\txi{0,1,1}(x^2q^{2}\txi{1,1,1}\txi{0,0,0}-a_1a_2\txi{0,0,1}\txi{1,1,0})}{\txi{0,0,1}(a_1\txi{0,1,0}\txi{1,1,1}-a_1a_2\txi{1,1,0}\txi{0,1,1})+\txi{0,0,0}\txi{1,1,1}\txi{0,1,1}}\\
 &=
 \frac{\txi{0,1,0}}{\txi{0,0,0}}\frac{-a_1a_2^2xq\txi{0,1,1}\txi{1,0,1}+a_1x^2q^{2}\txi{0,0,1}\txi{1,1,1}}{\txi{1,1,1}\txi{0,1,1}-a_1^2a_2^2xq\txi{0,0,1}\txi{1,2,1}}
 =
 \frac{\txi{0,1,0}\txi{0,0,2}}{\txi{0,0,0}\txi{0,1,2}}=\text{{\rm LHS} of (\ref{Tf1})}. 
\end{align*}

Using equations $(\ref{13}), (\ref{17})$ and $(\ref{18})$, we have

\begin{align*}
& \text{{\rm RHS} of (\ref{Tf2})}
 =
 \frac{\txi{1,1,1}\txi{0,0,0}\txi{0,1,1}+\txi{0,0,1}(a_1\txi{0,1,0}\txi{1,1,1}-a_1a_2\txi{1,1,0}\txi{0,1,1})}{x^2q^2\txi{1,1,1}\txi{0,0,0}\txi{0,1,1}+\txi{0,0,1}(a_1\txi{1,1,0}\txi{0,1,1}-a_1a_2\txi{1,1,0}\txi{0,1,1})}\\
 &=
 \frac{\txi{1,1,1}\txi{0,1,1}-a_1^2a_2^2xq\txi{0,0,1}\txi{1,2,1}}{x^2q^2\txi{1,1,1}\txi{0,1,1}-a_1^2a_2^2xq\txi{0,0,1}\txi{1,2,1}}
 =
 \frac{1}{x^2q^2}\frac{\txi{1,1,0}\txi{0,1,2}}{\txi{0,1,0}\txi{1,1,2}}=\text{{\rm LHS} of (\ref{Tf2})}.
\end{align*}

Using equations $(\ref{11})$ and $(\ref{12})$, we have

\begin{align*}
& \text{{\rm RHS} of (\ref{t1f1})}
 =
 \frac{x^2q^2\txi{1,1,1}\txi{0,0,0}-a_1a_2\txi{0,0,1}\txi{1,1,0}}{a_1a_2x^2\txi{0,0,0}\txi{1,1,1}-\txi{0,0,1}\txi{1,1,0}}
 =
 -q\frac{\txi{1,0,1}\txi{0,1,0}}{\txi{0,1,1}\txi{1,0,0}}=\text{{\rm LHS} of (\ref{t1f1})}.
\end{align*}

Using equations $(\ref{11}), (\ref{12}), (\ref{15})$ and $(\ref{16})$, we have

\begin{align*}
\text{{\rm RHS} of (\ref{t1f2})}
 &=
 \frac{a_1\txi{0,1,1}\txi{1,1,0}(a_1a_2\txi{0,0,1}\txi{1,1,0}-x^2q^{2}\txi{0,0,0}\txi{1,1,1})-q^2\txi{0,1,0}\txi{1,1,1}(\txi{0,0,1}\txi{1,1,0}-a_1a_2x^2\txi{0,0,0}\txi{1,1,1})}{a_1\txi{0,1,0}\txi{1,1,1}(a_1a_2x^2\txi{0,0,0}\txi{1,1,1}-\txi{0,0,1}\txi{1,1,0})-\txi{1,1,0}\txi{0,1,1}(x^2q^2\txi{0,0,0}\txi{1,1,1}-a_1a_2\txi{0,0,1}\txi{1,1,0})}\\
 &=
 q\frac{a_1\txi{1,1,0}\txi{1,0,1}+q\txi{1,1,1}\txi{1,0,0}}{a_1\txi{1,1,1}\txi{1,0,0}+q\txi{1,1,0}\txi{1,0,1}}
 =
 \frac{1}{x^2}\frac{\txi{0,0,1}\txi{2,1,0}}{\txi{2,1,1}\txi{0,0,0}}
 =\text{{\rm LHS} of (\ref{t1f2})}.
\end{align*}

Using equations $(\ref{11}), (\ref{12}), (\ref{13})$ and $(\ref{14})$, we have

\begin{align*}
\text{{\rm RHS} of (\ref{t2f1})}
 &=
 \frac{a_2\txi{0,0,0}\txi{0,1,1}(a_1a_2\txi{0,0,1}\txi{1,1,0}-x^2q^{2}\txi{0,0,0}\txi{1,1,1})-q^2\txi{0,0,1}\txi{0,1,0}(a_1a_2x^2\txi{0,0,0}\txi{1,1,1}-\txi{0,0,1}\txi{1,1,0})}{a_2\txi{0,0,1}\txi{0,1,0}(\txi{0,0,1}\txi{1,1,0}-a_1a_2x^2\txi{0,0,0}\txi{1,1,1})-\txi{0,0,0}\txi{0,1,1}(x^2q^2\txi{0,0,0}\txi{1,1,1}-a_1a_2\txi{0,0,1}\txi{1,1,0})}\\
 &=
 q\frac{q\txi{0,0,1}\txi{1,0,0}-a_2\txi{0,0,0}\txi{1,0,1}}{a_2\txi{0,0,1}\txi{1,0,0}-q\txi{0,0,0}\txi{1,0,1}}
 =
 -\frac{1}{x^2}\frac{\txi{1,1,0}\txi{0,-1,1}}{\txi{1,1,1}\txi{0,-1,0}}=\text{{\rm LHS} of (\ref{t2f1})}.
\end{align*}

Using equations $(\ref{11})$ and $(\ref{12})$, we have

\begin{align*}
& \text{{\rm RHS} of (\ref{t2f2})}
 =
 \frac{a_1a_2x^2\txi{0,0,0}\txi{1,1,1}-\txi{0,0,1}\txi{1,1,0}}{x^2q^2\txi{1,1,1}\txi{0,0,0}-a_1a_2\txi{0,0,1}\txi{1,1,0}}
 =
 -\frac{1}{q}\frac{\txi{0,1,1}\txi{1,0,0}}{\txi{1,0,1}\txi{0,1,0}}=\text{{\rm LHS} of (\ref{t2f2})}.
\end{align*}
\end{proof}
\begin{proof}[Proof of Theorem \ref{thm:1}]
Putting $(a_1,a_2,x)=(-q^m,q^{-n},xq^k)$ in the bilinear relations (\ref{11})-(\ref{18}). Then the function $\tilde{\xi}_{0,0,0}=\det\left[\mu(u+v+k\tau,v;m-j-j')\right]_{j,j'=0}^n$ satisfies the bilinear equations $(\ref{11})$--$(\ref{18})$ from Theorem $\ref{bi linear}$. Hence a solution of the type $(A_2+A_1)^{(1)}$ $q$-Painlev\'{e} equation is expressed using the function $\tilde{\xi}_{m,n,k}$, and Theorem \ref{thm:1} holds. 
\end{proof}
The bilinear relations (\ref{11}) and (\ref{12}) are the relations for three pairs of opposite vertices of a cube, respectively. 
\begin{figure}[htbp]
  \begin{minipage}{0.45\linewidth}
    \centering
    \includegraphics[keepaspectratio, scale=0.15]{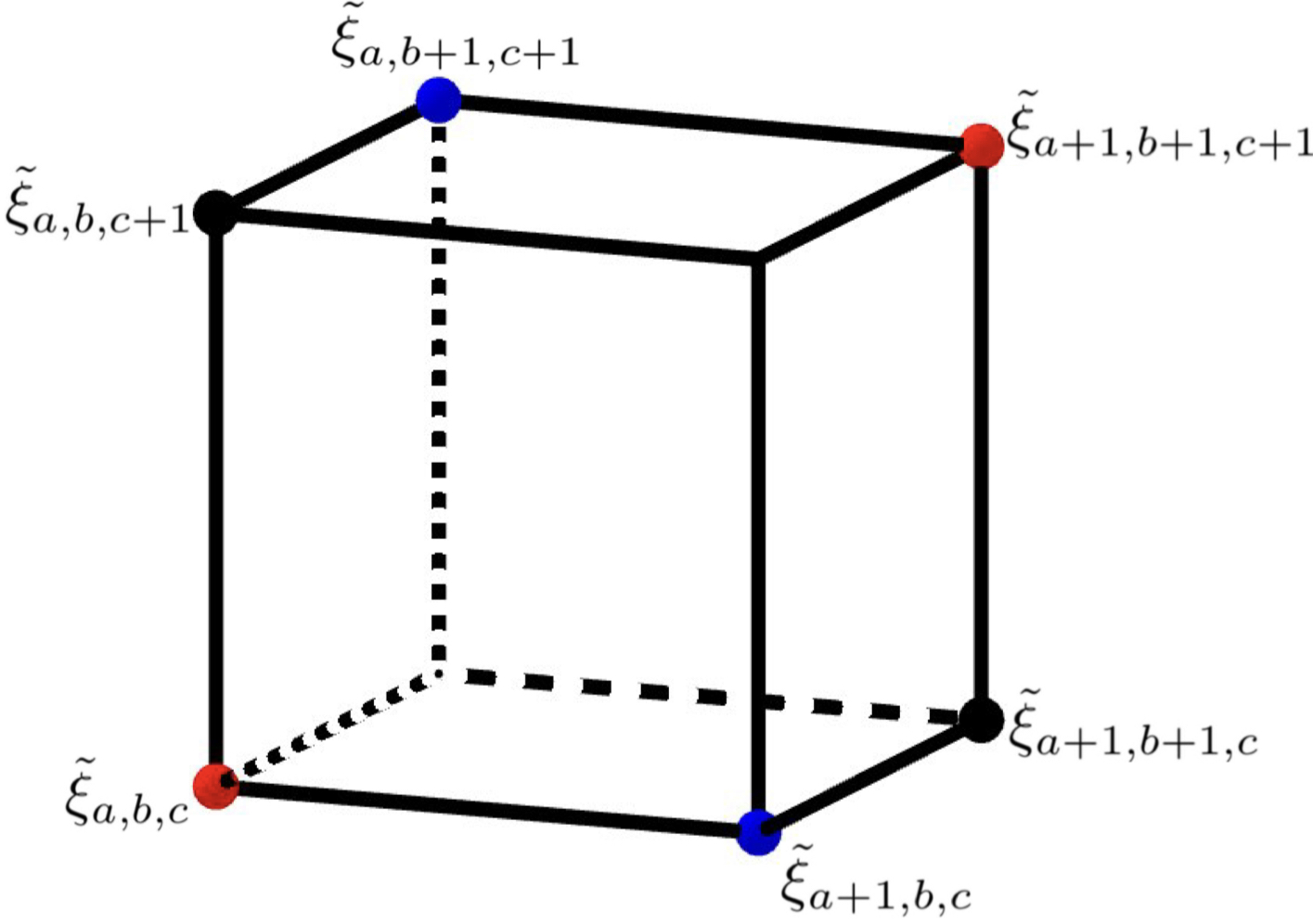}
    \caption{Equation (\ref{11})}
  \end{minipage}
  \begin{minipage}{0.45\linewidth}
    \centering
    \includegraphics[keepaspectratio, scale=0.135]{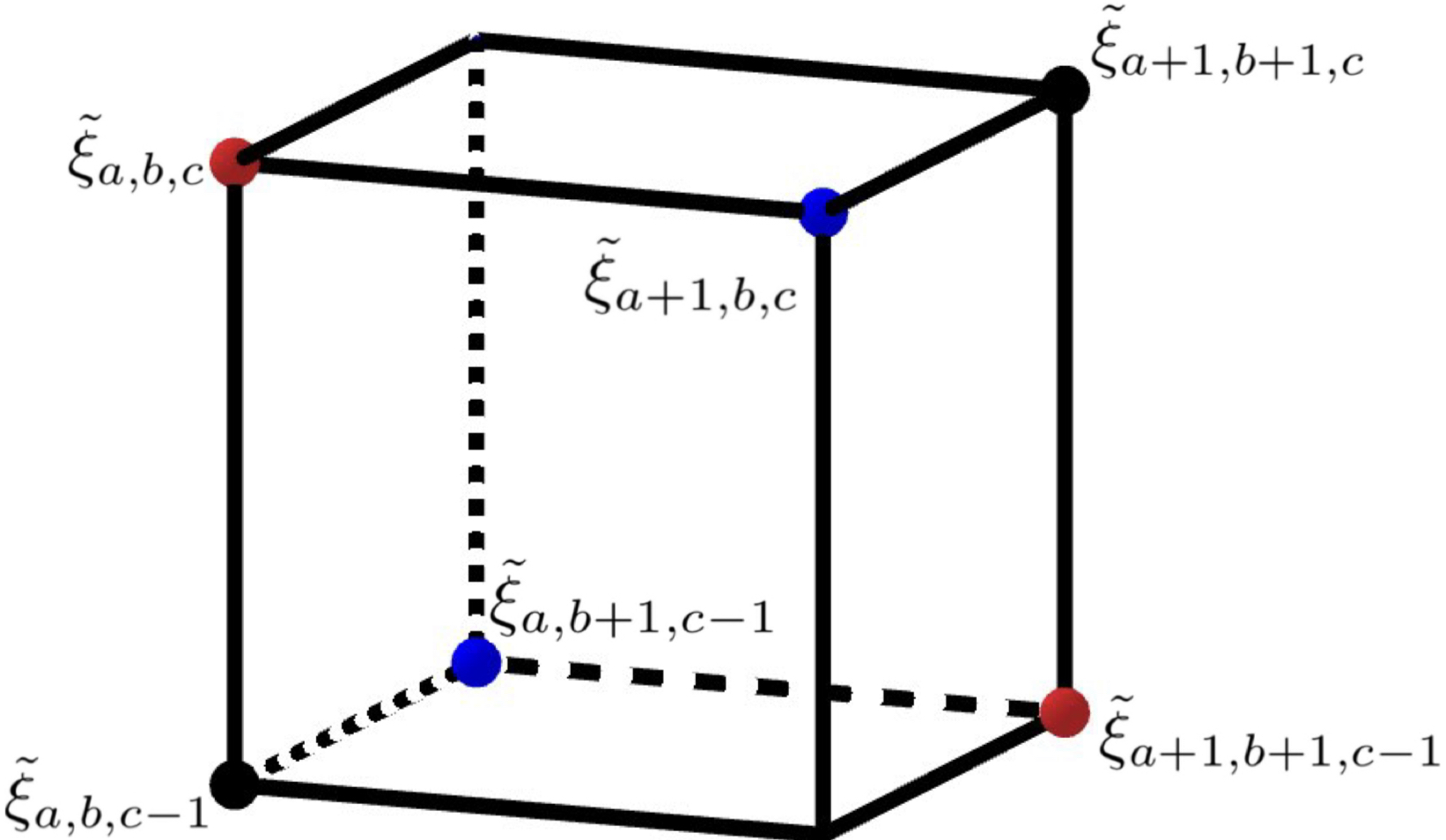}
    \caption{Equation (\ref{12})}
  \end{minipage}
\end{figure}
From these relations, we obtain bilinear relations for any three pairs of opposite vertices of a cube. 
That is, if we give a set of initial values for the six vertices which include all the vertices of a certain face of a cube, we obtain the value of any vertices of the cube. 
From The bilinear relations (\ref{13})-(\ref{18}), we obtain the values of two vertices of a cube whose faces are adjacent to the original cube. 
Therefore, if we give the initial values, the values of any $\xi_{m,n,k}$ is uniquely determined. 
In particular, the solution (\ref{solution}) corresponds to the initial values $\tilde{\xi}_{0,n+1,0}=\tilde{\xi}_{0,n+1,1}=\tilde{\xi}_{1,n+1,0}=\tilde{\xi}_{1,n+1,1}=1, \tilde{\xi}_{0,n,0}=\mu(u+v+k\tau,v;m)$ and $\tilde{\xi}_{0,n,1}=\mu(u+v+(k+1)\tau,v;m)$. 
\section*{Acknowledgments}
The author would like to thank professor Yasuhiko Yamada (Kobe University) for his helpful advice on this paper. He would like to thank professor Genki Shibukawa (Kobe University) for helpful discussion and guidance. He also wish to thank professor Yasuhiro Ohta (Kobe University), Tetsu Masuda (Aoyama Gakuin University) and Hidehito Nagao (Akashi College) for valuable suggestions on $q$-Painlev\'{e} equations.

\medskip
\begin{flushleft}
Satoshi Tsuchimi\\
Department of Mathematics\\
Graduate School of Science\\ 
Kobe University\\
1-1, Rokkodai, Nada-ku\\ 
Kobe, 657-8501\\
JAPAN\\
183s014s@stu.kobe-u.ac.jp
\end{flushleft}
\end{document}